\documentclass[11pt,reqno]{amsart}
\usepackage{amssymb,latexsym}
\usepackage{graphicx}
\usepackage{fancyhdr}
\usepackage{amsthm}
\usepackage{mathtools}
\usepackage[hyphens]{url}
\usepackage{hyperref}
\usepackage{breakurl}
\usepackage{cite}

\newcommand{\Z}{\mathbb Z}
\newcommand{\Q}{\mathbb Q}

\theoremstyle{plain}
\numberwithin{equation}{section}
\newtheorem{thm}{Theorem}[section]
\newtheorem{theorem}[thm]{Theorem}
\newtheorem{lemma}[thm]{Lemma}
\newtheorem{proposition}[thm]{Proposition}
\newtheorem{conjecture}[thm]{Conjecture}

\newtheorem*{mainconjecture}{Main Conjecture}
\newtheorem*{ranktheorem}{Rank Theorem}

\theoremstyle{definition}
\newtheorem{definition}[thm]{Definition}

\newtheorem{remark}[thm]{Remark}

\begin{document}

\setcounter{page}{1}

\title[Class numbers in cyclotomic $\Z_p$-extensions]{Class numbers in cyclotomic $\Z_p$-extensions}
\author{John C. Miller}
\address{Department of Mathematics\\
                Rutgers University\\
                Hill Center for the Mathematical Sciences\\
                110 Frelinghuysen Road
                Piscataway, NJ 08854-8019}
\email{jcmiller@math.rutgers.edu}
\thanks{This material is based upon work supported by the National Science Foundation under Grant No. DUE-1022574.}


\begin{abstract}
For any prime $p$ and any positive integer $n$,  let $\mathbb{B}_{p, n}$ denote the $n$th layer of the cyclotomic $\mathbb{Z}_p$-extension over the rationals.  Based on the heuristics of Cohen and Lenstra, and refined by new results on class numbers of particular $\mathbb{B}_{p, n}$, we provide evidence for the following conjecture:  For \emph{all} primes $p$ and positive integers $n$, the ideal class group of $\mathbb{B}_{p, n}$ is trivial.\end{abstract}

\maketitle

\section{Introduction}
Let $\mathbb{B}_{p, n}$ denote the $n$th layer of the cyclotomic $\mathbb{Z}_p$-extension over the rationals, i.e. the unique real subfield of the cyclotomic field $\mathbb{Q}(\zeta_{p^{n+1}})$ of degree $p^n$ over $\mathbb{Q}$ for odd primes $p$, and $\mathbb{Q}(\cos(2\pi/2^{n+2}))$ for $p=2$.  The class numbers of these fields have been the objects of intense investigation by number theorists, starting perhaps with Weber, who studied the $\mathbb{Z}_2$-extension.  A key step in his proof of the Kronecker-Weber theorem is the following theorem \cite{Weber}.

\begin{theorem}[Weber]
For all positive integers $n$, the class number of $\mathbb{B}_{2, n}$ is odd.
\end{theorem}

Weber's result was later generalized by Iwasawa \cite{Iwasawa} to all cyclotomic $\mathbb{Z}_p$-extensions over $\mathbb{Q}$.
\begin{theorem}[Iwasawa]\label{Iwasawa}
For all positive integers $n$ and primes $p$, the class number of $\mathbb{B}_{p, n}$ is not divisible by $p$, i.e.\ the $p$-part of the class group is trivial.
\end{theorem}
However, the non-$p$-part of the class groups of the $\mathbb{Z}_p$-extensions remains deeply mysterious.  The only general result we have is Washington's ``local" boundedness result \cite{Washington_non_p}.
\begin{theorem}[Washington]
Let $q$ be a prime not equal to $p$, and let $q^{e_n}$ be the exact power of $q$ dividing the class number of $\mathbb{B}_{p, n}$.  Then $e_n$ is constant for sufficiently large $n$.
\end{theorem}

On the other hand, much recent progress has been made regarding the $\mathbb{Z}_2$- and $\mathbb{Z}_3$-extensions. By analyzing properties of the cyclotomic units, Fukuda and Komatsu \cite{Fukuda} went much further than Weber and proved the following striking result.

\begin{theorem}[Fukuda, Komatsu]
For all positive integers $n$, the class number of $\mathbb{B}_{2, n}$ is not divisible by any prime less than $10^9$.
\end{theorem}

This suggests the following conjecture, known as \emph{Weber's class number problem}.

\begin{conjecture}[Weber's class number problem]
For all positive integers $n$, the class number of $\mathbb{B}_{2, n}$ is $1$.
\end{conjecture}

Fukuda, Komatsu and Morisawa proved a similar result for the $\mathbb{Z}_3$-extension \cite{Fukuda2}, which suggests that all the $\mathbb{B}_{3, n}$ too may have trivial class group.
\begin{theorem}[Fukuda, Komatsu, Morisawa]
For all positive integers $n$, the class number of $\mathbb{B}_{3, n}$ is not divisible by any prime less than $10^9$.
\end{theorem}

Extending Weber's theorem in a different direction, Ichimura and Nakajima \cite{Ichimura} proved a result concerning the parity of class numbers.
\begin{theorem}[Ichimura, Nakajima]\label{ParityClassNumber}
For all positive integers $n$ and primes $p < 500$, the class number of $\mathbb{B}_{p, n}$ is odd.
\end{theorem}

Also, in unpublished work by T. Fukuda and his collaborators, they have calculated that the prime $2p+1$ does not divide the class number of $\mathbb{B}_{p, 1}$ for Sophie Germain primes $p$ less than 1000 (T. Fukuda, private communication, July 25, 2014). 

Notwithstanding the considerable interest in $\mathbb{Z}_p$-extensions and their class numbers, the precise class number has been successfully determined for only a very few of the $\mathbb{B}_{p, n}$, due to the inherent difficulties of calculating the class group for number fields of large discriminant and degree.  The previously known results are summarized in Table \ref{ClassNumbers}.  All the known class numbers are 1.

\begin{center}
\begin{table}
\caption{$\mathbb{B}_{p, n}$ with previously known class numbers}
\begin{tabular}{@{}  l l l @{}}
\hline
  $\mathbb{B}_{p, n}$ & \bf{Class number} &\\
\hline
   $\mathbb{B}_{2, 3}$     	& 1      			& Cohn (1960) \cite{Cohn}\\
   $\mathbb{B}_{2, 4}$     	& 1      			& Bauer (1969) \cite{Bauer}\\
   $\mathbb{B}_{2, 5}$     	& 1      			& van der Linden (1982) \cite{Linden}\\
   $\mathbb{B}_{2, 6}$     	& 1      			& Miller (2013) \cite{MillerWeber}\\
   $\mathbb{B}_{2, 7}$     	& 1 (under GRH)     	& Miller (2013)\\
   $\mathbb{B}_{3, 2}$     	& 1      			& Bauer (1969)\\
   $\mathbb{B}_{3, 3}$     	& 1      			& Masley (1978) \cite{Masley}\\
   $\mathbb{B}_{3, 4}$     	& 1 (under GRH)      	& van der Linden (1982)\\
   $\mathbb{B}_{5, 1}$     	& 1      			& Bauer (1969)\\
   $\mathbb{B}_{5, 2}$     	& 1      			& Miller (2014) \cite{MillerComp}\\
   $\mathbb{B}_{7, 1}$     	& 1      			& Bauer (1969)\\
   $\mathbb{B}_{11, 1}$     	& 1      			& Miller (2014)\\
   $\mathbb{B}_{13, 1}$     	& 1 (under GRH)      	& van der Linden (1982)\\
   $\mathbb{B}_{17, 1}, \mathbb{B}_{19, 1}, \mathbb{B}_{23, 1}$     	& 1 (under GRH)      	& Coates, Liang, Mih\v ailescu (2012) \cite{Coates}\\
\hline
\end{tabular}
\label{ClassNumbers}
\end{table}
\end{center}

By using the analytic upper bound for the class number developed in \cite{MillerWeber}, we will prove unconditionally that several more fields have class number 1.

\begin{theorem}\label{NewClassNumbers}
The fields $\mathbb{B}_{13, 1}$, $\mathbb{B}_{17, 1}$ and $\mathbb{B}_{19, 1}$ have class number $1$.
\end{theorem}

Furthermore, using SageMathCloud\textsuperscript{TM} \cite{SageMathCloud}, which calls underlying routines in the Pari library \cite{Pari}, we calculated that the fields $\mathbb{B}_{29, 1}$ and $\mathbb{B}_{31, 1}$ each have class number 1, under the assumption of the generalized Riemann hypothesis.  We summarize these results in Table \ref{ClassNumbersNew}.

\begin{center}
\begin{table}
\caption{$\mathbb{B}_{p, n}$  with new class number results}
\begin{tabular}{@{}  l l l  @{}}
\hline
  $\mathbb{B}_{p, n}$ & \bf{Class number} &\\
\hline
   $\mathbb{B}_{13, 1}, \mathbb{B}_{17, 1}, \mathbb{B}_{19, 1}$     	& 1      	& Theorem \ref{NewClassNumbers}\\
   $\mathbb{B}_{29, 1}, \mathbb{B}_{31, 1}$     	& 1 (under GRH)      	& Pari via SageMathCloud\textsuperscript{TM}\\
\hline
\end{tabular}
\label{ClassNumbersNew}
\end{table}
\end{center}

Coates \cite{Coates} posed the following natural question: Do all $\mathbb{B}_{p, n}$ have class number 1?  Unfortunately, resolving such a question is far beyond our current capabilities.  Indeed, even proving that there are infinitely many number fields with class number 1 is one of the great open problems of number theory.  On the other hand, the class group heuristics developed by Cohen and Lenstra \cite{CohenLenstra1, CohenLenstra2} provide a means by which we can attempt to estimate the probability that one or more of the $\mathbb{B}_{p, n}$ have nontrivial class groups.

The first steps in this direction were carried out by Buhler, Pomerance and Robertson \cite{Buhler}.  By proving a certain quadruple sum is finite, they were led to the following conjecture.

\begin{conjecture}[Buhler, Pomerance, Roberston]
For all but finitely many pairs $(p,n)$, where $p$ is a prime and $n$ is a positive integer, the class number of the real cyclotomic field of conductor $p^{n+1}$ is equal to the class number of the real cyclotomic field of conductor $p^n$, i.e.
$$h\left(\mathbb{Q}(\zeta_{p^{n+1}}+\zeta_{p^{n+1}}^{-1})\right) = h\left(\mathbb{Q}(\zeta_{p^n}+\zeta_{p^n}^{-1})\right).$$
\end{conjecture}

A consequence of this conjecture is that the class number of $\mathbb{B}_{p, n}$ is equal to the class number of $\mathbb{B}_{p, n-1}$ for all but finitely many pairs $(p,n)$.

In this paper, we will study a similar quadruple sum and find an \emph{explicit} upper bound for the sum, which essentially estimates the expected number of exceptional pairs $(p,n)$.  This upper bound is then further refined by known results on class numbers of the $\mathbb{B}_{p, n}$, including the new results established in this paper.  In fact, this refined upper bound is approximately $0.30$, which is sufficiently smaller than 1 to suggest the following conjecture.

\begin{mainconjecture}\label{Main Conjecture}
For any prime $p$ and positive integer $n$, the class number of $\mathbb{B}_{p, n}$ is 1.
\end{mainconjecture}

In order to rigorously establish an upper bound for the quadruple sum, an explicit sieve theory result of Siebert  \cite{Siebert} and an improvement of Riesel and Vaughan \cite{RieselVaughan} will prove vital.  Along the way, we will also prove an upper bound for the sum of the reciprocals of the Sophie Germain primes, which surprisingly does not yet exist in the literature.

The rest of the paper is organized as follows:  In Section 2 we will discuss upper bounds on class numbers and prove Theorem \ref{NewClassNumbers}.  The application of Cohen-Lenstra heuristics to real cyclic number fields will be discussed in detail in Section 3.  In Section 4 we specialize the Cohen-Lenstra heuristics to $\mathbb{Z}_p$-extensions.  Sections 5, 6 and 7 study the $\mathbb{Z}_2$-, $\mathbb{Z}_3$- and $\mathbb{Z}_5$-extensions respectively, and Section 8 concerns the $\mathbb{Z}_p$-extensions for $7 \leq p \leq 31$.  Section 9 discusses the first layers of $\mathbb{Z}_p$-extensions.  In Section 10, we introduce the explicit sieve results and apply them to bounds on sums over Sophie Germain and similar primes.  Using those bounds, in Section 11 we return to the discussion of the first layers.  We study the higher layers in Section 12. The results are summarized in Section 13.

\section{New results on certain $\mathbb{B}_{p, n}$}
In his 1982 paper \cite{Linden}, van der Linden employed Odlyzko's discriminant lower bounds \cite{OdlyzkoTable} to prove, conditionally upon the generalized Riemann hypothesis (GRH), that the class number of $\mathbb{B}_{13,1}$ is 1.  However, the root discriminant of $\mathbb{B}_{13,1}$ is too large for the class number to be calculated unconditionally using van der Linden's methods, and the field $\mathbb{B}_{17,1}$ has root discriminant too large to be treated by those methods, even under the assumption of GRH.
In the author's recent paper \cite{MillerWeber}, an analytic class number upper bound was developed that employed counting prime ideals of the Hilbert class field.  Using this new bound, it is possible to unconditionally determine the class number of number fields of discriminant too large to have been treated by previously known methods.

\begin{definition}
Let $K$ denote a number field of degree $n$ over $\mathbb Q$.  Let $d(K)$ denote its discriminant.  The \emph{root discriminant} $\operatorname{rd}(K)$ of $K$ is defined to be:
\[\operatorname{rd}(K)=|d(K)|^{1/n}.\]
\end{definition}

We have the following unconditional upper bound for class numbers.

\begin{lemma}[Miller \cite{MillerWeber}]\label{SplitPrimeLemmaNoRH}
Let $K$ be a totally real field of degree $n$, and let
$$F(x) = \frac{e^{-\left(x/c\right)^2}}{\cosh(x/2)}$$
for some positive constant $c$.  Suppose $S$ is a subset of the primes which totally split into principal prime ideals of $K$.  Let
\begin{align*}
B = \frac{\pi}{2} &+ \gamma + \log{8\pi}  - \log \operatorname{rd}(K) - \int_0^\infty \frac{1-F(x)}{2}\left( \frac{1}{\sinh(x/2)} + \frac{1}{\cosh(x/2)} \right) \, dx \\&+ 2 \sum_{p \in S} \sum_{m=1}^\infty \frac{\log p}{p^{m/2}}F(m \log p),
\end{align*}
where $\gamma$ is Euler's constant.  If $B > 0$, then the class number $h$ of $K$ has an upper bound
$$h < \frac{2c \sqrt{\pi}}{nB}.$$
\end{lemma}

All the terms in the above lemma can be calculated explicitly, including the integral which can be estimated by numerical integration.  The strategy is to find sufficiently many generators of totally split principal prime ideals, so that we have sufficiently many primes to include in our set $S$.  Then we can apply the lemma and establish an upper bound for the class number.  The reader is invited to consult \cite{MillerWeber} for more details regarding this method.

After finding upper bounds, we will unconditionally prove that $\mathbb{B}_{13,1}$, $\mathbb{B}_{17,1}$ and $\mathbb{B}_{19,1}$ have class number 1.

To find a generator of a principal ideal, it is useful to have an integral basis.  The field $\mathbb{B}_{p,1}$ is the degree $p$ subfield of the cyclotomic field of conductor $p^2$, which is cyclic.  If $\sigma$ generates the Galois group of the cyclotomic field, then the Galois subgroup generated by $\sigma^p$ fixes the subfield $\mathbb{B}_{p,1}$.  Let $\zeta = \exp(2\pi i/p^2)$.  Given the generator $\sigma$, we'll use $\{b_0, b_1, \dots, b_{p-1}\}$ as our integral basis, where $b_0 = 1$, and
$$b_j = \sum_{k=0}^{p-2} \zeta^{\sigma^{kp+j-1}}$$
for $1 \leq j \leq p-1$.

After using Lemma \ref{SplitPrimeLemmaNoRH} to establish an upper bound for the class number, we will use the following version of the Rank Theorem (see Masley \cite{Masley}, Cor 2.15) to pin down the exact class number.

\begin{ranktheorem}\label{Rank Theorem}
Let $K$ be a number field of prime degree $p$ over the rationals.  If a prime $q$ not equal to $p$ divides the class number of $K$, then $q^f$ divides the class number, where $f$ is the order of $q$ modulo $p$.
\end{ranktheorem}

Now we will prove our results for $\mathbb{B}_{13,1}$, $\mathbb{B}_{17,1}$ and $\mathbb{B}_{19,1}$.

\begin{proposition}
The class number of $\mathbb{B}_{13,1}$ is $1$.
\end{proposition}

\begin{proof}
The root discriminant of $\mathbb{B}_{13,1}$ is approximately 113.90, which is too large to use Odlyzko's unconditional discriminant bounds.  Alternatively, using the Minkowski bound of 479,001,600 would require testing whether roughly 20 million prime ideals were principal.  However, using our new method, it suffices for us to find generators for \emph{only two} principal prime ideals in order to employ Lemma \ref{SplitPrimeLemmaNoRH} and establish an unconditional class number upper bound.

The Galois group of the cyclotomic field of conductor 169 is generated by $\sigma = (\zeta \mapsto \zeta^2)$, where $\zeta = \exp(2 \pi i/169)$.  Using the integral basis $\{b_0, b_1, \dots, b_{12}\}$ given above, we search over sparse vectors and find that the element
$$b_0 - b_1 - b_2 - b_3 - b_6 - b_7 - b_8$$
has norm 19, where the \emph{norm} is taken to be the absolute value of the product of the Galois conjugates of the element.  Similarly, the element
$$b_0 - b_1 - b_2 -b_4 + b_{12}$$
has norm 23.  Therefore, the primes 19 and 23 totally split into principal ideals.

Now by Lemma \ref{SplitPrimeLemmaNoRH}, using $S = \{19, 23\}$ and $c = 10$, we get an upper bound for the class number,
$$h \leq 7.$$
Finally, by the Rank Theorem, we conclude that the class number is 1.
\end{proof}

\begin{proposition}
The class number of $\mathbb{B}_{17,1}$ is $1$.
\end{proposition}

\begin{proof}
The fields $\mathbb{B}_{p,1}$ have root discriminant $p^{2(1-1/p)}$, so already $\mathbb{B}_{17,1}$ will require finding many thousands of principal prime ideals, rather the just two as was the case for $\mathbb{B}_{13,1}$.  To find these primes, we will search over a large number of vectors, using an integral basis as given above.  In particular, we consider the vectors
\[x = a_1 b_{j_1} + a_2 b_{j_2} + \cdots + a_{12} b_{j_{12}},\]
where $0 \leq j_1 < j_2 < j_3 < \cdots < j_{12} < 17$ and $a_j \in \{-1, 0, 1\}$ for $1 \leq j \leq 12$,

Let $T$ denote the set of all such elements $x$, and let $U$ be the set of their norms, up to $10^{12}$.
\[U = \{N(x) | x \in T, N(x) < 10^{12} \}.\]
Let $S_1$ be the set of prime norms 
\[S_1 = \{p : p \in U, p \text{ prime}\}.\]

For a field of such large root discriminant, it is difficult to find sufficiently many elements of prime norm.  An effective approach is to search for elements that have large composite norms, and then take quotients by appropriately chosen algebraic integers.  Following this strategy, we define $S_2$ to be the set of primes defined by
\[S_2 = \{p : pq \in U, p \text{ prime, } p \notin S_1, q \in S_1  \},\]
noting that if $N(x) = pq$ and $N(y)=q$, for $x, y$ in the ring of integers $\mathcal{O}$, then $x/y^\sigma \in \mathcal{O}$ with norm $p$ for some Galois automorphism $\sigma$.

Now put $S = S_1 \cup S_2 \setminus \{17\}$ and $c = 22$.   The contribution of prime ideals has a lower bound
\[2 \sum_{p \in S} \sum_{m=1}^\infty \frac{\log p}{p^{m/2}}F(m \log p) > 2 \sum_{p \in S} \sum_{m=1}^2 \frac{\log p}{p^{m/2}}F(m \log p) > 2.092893.\]
Applying Lemma \ref{SplitPrimeLemmaNoRH}, we find that the class number is bounded above by
$$h \leq 5.$$
Therefore, by the Rank Theorem, we see that the class number is 1.
\end{proof}

\begin{proposition}
The class number of $\mathbb{B}_{19,1}$ is $1$.
\end{proposition}

\begin{proof}
We follow an entirely similar proof to that of $\mathbb{B}_{17,1}$ to generate our set $S$.  Using $c=40$, we calculate a lower bound for the contribution of prime ideals of 1.671648.  We apply Lemma \ref{SplitPrimeLemmaNoRH} and find that the class number is bounded above by
$$h \leq 38.$$
By Theorem \ref{Iwasawa}, the class number is not divisible by 19.  For the other primes $p \neq 19$, we can apply the Rank Theorem to conclude that the class number is 1.
\end{proof}

\section{Cohen-Lenstra heuristics for totally real cyclic number fields}
If a number field is a Galois extension of the rationals, then its Galois group $G$ acts on its ideals, so its ideal class group is a $\mathbb{Z}[G]$-module.  Moreover, if
$$N = \sum_{g \in G} g,$$
then $N$ acts as the norm on ideals, sending every ideal to a principal ideal, so that the class group is in fact a $\mathbb{Z}[G]/\langle N \rangle$-module.  Roughly speaking, Cohen and Lenstra \cite{CohenLenstra1, CohenLenstra2} predicted that the class group of a totally real Galois number field $K$ of Galois group $G$ should behave as a ``random" finite $\mathbb{Z}[G]/\langle N \rangle$-module modulo a random cyclic submodule.  One can then ask what is the expected probability that a simple factor $M$ occurs in the Jordan-H\"{o}lder decomposition of this quotient.  When $G$ is abelian, Cohen and Lenstra calculated that a random $\mathbb{Z}[G]/\langle N \rangle$-module modulo a random cyclic submodule should have such a simple factor $M$ with probability:
$$1 - \prod_{k \geq 2} (1 - |M|^{-k}).$$

Here we are taking some license with our notion of ``probability."  In Cohen and Lenstra's formulation of their heuristics, probability has a precise meaning in terms of the density of number fields of a given Galois group and signature, as the discriminant varies.  Our notion of probability will be in the Bayesian sense of a subjective prior probability, as described by Buhler, Pomerance and Robertson \cite[p. 3]{Buhler}:
\begin{quote}
We would like to apply this to class groups $Cl^+(\ell^n)$, but this is hard to formalize in the usual frequentist language of probability since there is no underlying probability space.  Indeed, the original Cohen-Lenstra heuristics apply to a large collection of fields of a given degree, and we are applying them to a large collection of fields whose degrees are unbounded.  Instead we adopt a subjective Bayesian view, where probability arises from ignorance.  Thus we use the Cohen-Lenstra heuristics as the basis of the assignment of subjective probabilities, on the grounds that they are a plausible first guess.
\end{quote}
Schoof also used such a notion of probability to analyze his results on the class numbers of real cyclotomic fields of prime conductor \cite[p. 19]{Schoof}.  Going forward, when we speak of the ``probability" of a number field (or family of number fields) having class number 1, we mean probability in this Bayesian sense.

With this stipulation on the meaning of ``probability," given a prime $q$ not dividing the degree of $K$, the Cohen-Lenstra heuristics predict the probability that $q$ does not divide the class number $h(K)$ of $K$ is
$$\operatorname{Prob}(q\text{-part of } h(K) \text{ is } 1) = \prod_{\substack{M \text{ simple} \\ |M|\text{ power of }q}}  \prod_{k \geq 2} (1 - |M|^{-k}),$$
where the outer product runs over all simple $\mathbb{Z}[G]/\langle N \rangle$-modules $M$ of order a power of $q$.

Now suppose $K$ is a totally real degree $n$ cyclic extension of the rationals, and suppose $q$ is a prime not dividing $n$.  Then the simple $\mathbb{Z}[G]$-modules $M$ of size power of $q$ are given by quotients of $\mathbb{Z}[G] \cong \mathbb{Z}[X]/(X^n-1)$ by the maximal ideals generated by $q$ and $f(X)$, where $f(X)$ is an irreducible divisor of the cyclotomic polynomial $\Phi_d(X)$ modulo $q$, for some $d$ dividing $n$.  To get the maximal ideals of  $\mathbb{Z}[G]/\langle N \rangle$, we further require that $d > 1$.  In other words, the simple $\mathbb{Z}[G]/\langle N \rangle$-modules $M$ of size power of $q$ are given by 
$$M \cong \mathbb{Z}[\zeta_d]/P$$
for some divisor $d$ of $n$, with $d>1$, and for some prime ideal $P$ lying over $q$.

\begin{remark}
Although we have the isomorphism,
$$\mathbb{Z}[G]/\langle N \rangle \cong \frac{\mathbb{Z}[X]}{(1+X+\cdots+X^{n-1})} = \frac{\mathbb{Z}[X]}{\left(\prod_{d|n,\, d>1}\Phi_d(X)\right)},$$
the homorphism
$$\frac{\mathbb{Z}[X]}{\left(\prod_{d|n,\, d>1}\Phi_d(X)\right)} \longrightarrow \bigoplus_{d|n,\, d>1} \frac{\mathbb{Z}[X]}{\left(\Phi_d(X)\right)} \cong \bigoplus_{d|n, \, d>1} \mathbb{Z}[\zeta_d],$$
is injective, but is not in general surjective.  Nevertheless, the maximal ideals relatively prime to $n$ are in bijective norm-preserving correspondence.
\end{remark}

Therefore, given a prime $q$ not dividing the degree $n$, we have the heuristic probability
$$\operatorname{Prob}(q\text{-part of } h(K) \text{ is } 1) = \prod_{d|n, \, d>1} \prod_{\substack{P \subset \mathbb{Z}[\zeta_d] \\ P|q}}  \prod_{k \geq 2} (1 - NP^{-k}),$$
where the $P$ are prime ideals of $\mathbb{Z}[\zeta_d]$ that lie over $q$, and $NP$ is the norm of $P$.  If the primes $P$ over $q$ have degree $f$, or equivalently if the order of $q$ modulo $d$ is $f$, then there are $\phi(d)/f$ primes $P$ lying over $q$, where $\phi$ is the Euler totient function.  So we can write
$$\operatorname{Prob}(q\text{-part of } h(K) \text{ is } 1) = \prod_{d|n, \, d>1} \prod_{k \geq 2} (1 - q^{-fk})^{\phi(d)/f}.$$

In particular, for a totally real cyclic field $K$ of prime power degree $p^n$, the Cohen-Lenstra heuristics predict that
$$\operatorname{Prob}(\text{non-}p\text{-part of } h(K) \text{ is } 1) = \prod_{\substack{q \text{ prime} \\ q \neq p}} \prod_{j=1}^n \prod_{k \geq 2} (1 - q^{-fk})^{\phi(p^j)/f},$$
where $f$ is the order of $q$ modulo $p^j$.

\section{Cohen-Lenstra heuristics for cyclotomic $\mathbb{Z}_p$-extensions}
Since $\mathbb{B}_{p, n}$, the $n$th layer of the cyclotomic $\mathbb{Z}_p$-extension, is cyclic with prime power degree $p^n$, the heuristics for the non-$p$-part of the class group are described by the previous section.  Moreover, Iwasawa \cite{Iwasawa} proved Theorem \ref{Iwasawa}, i.e. that $p$ does not divide the class number of $\mathbb{B}_{p, n}$ for any prime $p$ and any positive integer $n$.  Therefore, we may predict that the probability that $\mathbb{B}_{p, n}$ has class number 1 is
$$\operatorname{Prob}(h(\mathbb{B}_{p, n})=1) = \prod_{\substack{q \text{ prime} \\ q \neq p}} \prod_{j=1}^n \prod_{k \geq 2} (1 - q^{-fk})^{\phi(p^j)/f}.$$

We may conclude that a naive estimate of the probability $P_T$ of our Main Conjecture being true is
$$P_T = \prod_{p \text{ prime}} \prod_{\substack{q \text{ prime} \\ q \neq p}} \prod_{j=1}^\infty \prod_{k \geq 2} (1 - q^{-fk})^{\phi(p^j)/f}.$$
It is convenient to linearize this product by taking logarithms:
$$-\log{P_T} = -\sum_{p \text{ prime}} \sum_{\substack{q \text{ prime} \\ q \neq p}} \sum_{j=1}^\infty \sum_{k \geq 2} \frac{\phi(p^j)}{f}\log(1 - q^{-fk}).$$

Buhler, Pomerance and Robertson \cite{Buhler} proved that this quadruple sum is finite.  Calculating an explicit upper bound for such a sum will require certain results from sieve theory, in particular theorems of Siebert \cite{Siebert} and Riesel and Vaughan \cite{RieselVaughan}.  We will further refine our estimates by incorporating existing knowledge of class numbers of fields in $\mathbb{Z}_p$-extensions, including the new results of Theorem \ref{NewClassNumbers}; for this it is useful to know that for $m < n$ the conditional probability that $h(\mathbb{B}_{p, n})=1$, given that $h(\mathbb{B}_{p, m})=1$, is
$$\prod_{\substack{q \text{ prime} \\ q \neq p}} \prod_{j=m+1}^n \prod_{k \geq 2} (1 - q^{-fk})^{\phi(p^j)/f}.$$

As a warm-up, we first estimate some important special cases of our Main Conjecture.

\section{The $\mathbb{Z}_2$-extension}
The layers of the $\mathbb{Z}_2$-extensions are precisely the real cyclotomic fields with power of 2 conductor.  We have the following special case of our Main Conjecture, known as \emph{Weber's class number problem}:

\begin{conjecture}[Weber's class number problem]
For any positive integer $n$, the class number of $\mathbb{B}_{2, n} = \mathbb{Q}(\zeta_{2^{n+2}}+\zeta_{2^{n+2}}^{-1})$ is 1.
\end{conjecture}
Recently, the author \cite{MillerWeber} unconditionally proved that the class number of $\mathbb{B}_{2, 6}$ is 1 and also (under GRH) that the class number of $\mathbb{B}_{2, 7}$ is 1.  Furthermore, Fukuda and Komatsu \cite{Fukuda} have proven, for all primes $q < 10^9$, that the $q$-part of the class number of $\mathbb{B}_{2, n}$ is trivial, for all $n$.  Thus we may estimate the probability $P_T$ that the Weber conjecture is true by:
$$P_T = \prod_{\substack{q \text{ prime} \\ q > 10^9 }} \prod_{j=8}^\infty \prod_{k \geq 2} (1 - q^{-fk})^{\phi(2^j)/f}.$$
The probability $P_F$ that the Weber conjecture is false is bounded above by
$$P_F < -\log(1-P_F) = - \log{P_T} = - \sum_{\substack{q \text{ prime} \\ q > 10^9}} \sum_{j=8}^\infty \sum_{k \geq 2} \frac{\phi(2^j)}{f}\log(1 - q^{-fk}).$$
To estimate these probabilities, we have the following proposition.

\begin{proposition}
The following triple sum has an upper bound
$$- \sum_{\substack{q \,\mathrm{prime} \\ q > 10^9}} \sum_{j=8}^\infty \sum_{k \geq 2} \frac{\phi(2^j)}{f}\log(1 - q^{-fk}) < 1.3 \times 10^{-8},$$
where $\phi$ is the Euler totient function, and $f$ is the order of $q$ modulo $2^j$.
\end{proposition}

\begin{proof}
Let $S$ denote the triple sum.  Since $q^{-fk} < 10^{-18}$, we have
$$-\log(1 - q^{-fk}) < (1+10^{-17})q^{-fk}.$$
Therefore, the triple sum $S$ is bounded above by
$$S < (1+10^{-17}) \sum_{\substack{q \text{ prime} \\ q > 10^9}} \sum_{j=8}^\infty \frac{2^{j-1}}{fq^{f}(q^{f}-1)} < (1+10^{-8}) \sum_{\substack{q \text{ prime} \\ q > 10^9}} \sum_{j=8}^\infty \frac{2^{j-1}}{fq^{2f}}.$$
Recall that $f$ is the order of $q$ modulo $2^j$.  Following the argument in \cite[p. 6]{Buhler}, we observe that $q^f = 1 + 2^j t$ for some positive integer $t$, so we have the upper bound
$$S < (1+10^{-8}) \sum_{j=8}^\infty \sum_{\substack{t=1 \\ 1 + 2^j t > 10^9}}^\infty \frac{2^{j-1}}{(1 + 2^j t)^2} < (1+10^{-8}) \sum_{j=8}^\infty \sum_{\substack{t=1 \\ 1 + 2^j t > 10^9}}^\infty \frac{1}{2^{j+1}t^2}.$$
We can rearrange this to get
$$S  < \frac{1+10^{-8}}{2^8} \cdot \frac{\pi^2}{6} - \sum_{j=8}^\infty \sum_{\substack{t=1 \\ 1 + 2^j t \leq 10^9}}^\infty \frac{1}{2^{j+1}t^2} < 1.3 \times 10^{-8}.$$
\end{proof}

Thus, assuming the validity of the Cohen-Lenstra heuristics, the probability that the Weber conjecture is true is at least 99.999998\% (versus $> 99.3\%$ estimated in \cite{Buhler}).

\section{The $\mathbb{Z}_3$-extension}
The layers of the $\mathbb{Z}_3$-extensions are precisely the real cyclotomic fields with power of 3 conductor.  We have the following special case of our Main Conjecture.

\begin{conjecture}
For any positive integer $n$, the class number of $\mathbb{B}_{3, n} = \mathbb{Q}(\zeta_{3^{n+1}}+\zeta_{3^{n+1}}^{-1})$ is 1.
\end{conjecture}

Van der Linden \cite{Linden} proved that the class number of $\mathbb{B}_{3, 3}$ is 1 and also (under GRH) that the class number of $\mathbb{B}_{3, 4}$ is 1.  This conjecture is further supported by the work of Morisawa, who proved in his thesis \cite{Morisawa} that all primes $p$ less than $400,000$ do not the class number of $\mathbb{B}_{3, n}$ for all $n$.  Morisawa's result has recently been further improved by Fukuda, Komatsu and Morisawa \cite{Fukuda2}, who proved that no prime less than $10^9$ divides the class number of $\mathbb{B}_{3, n}$.

Thus we have we may estimate the probability $P_T$ that this conjecture is true by:
$$P_T = \prod_{\substack{q \text{ prime} \\ q > 10^9 }} \prod_{j=5}^\infty \prod_{k \geq 2} (1 - q^{-fk})^{\phi(3^j)/f}.$$
To estimate this probability, we have the following proposition.

\begin{proposition}
The following triple sum has an upper bound
$$- \sum_{\substack{q \,\mathrm{prime} \\ q > 10^9}} \sum_{j=5}^\infty \sum_{k \geq 2} \frac{\phi(3^j)}{f}\log(1 - q^{-fk}) < 1.1 \times 10^{-8},$$
where $\phi$ is the Euler totient function, and $f$ is the order of $q$ modulo $3^j$.
\end{proposition}

\begin{proof}
Let $S$ denote the triple sum.  Then we have
$$S < (1+10^{-8}) \sum_{\substack{q \text{ prime} \\ q > 10^9}} \sum_{j=5}^\infty \frac{\phi(3^j)}{fq^{2f}} <  (1+10^{-8}) \sum_{j=5}^\infty \sum_{\substack{t=1 \\ 1 + 3^j t > 10^9}}^\infty \frac{2}{3^{j+1}t^2}.$$
We rearrange this to get
$$S  < \frac{1+10^{-8}}{3^5} \cdot \frac{\pi^2}{6} - \sum_{j=5}^\infty \sum_{\substack{t=1 \\ 1 + 3^j t \leq 10^9}}^\infty \frac{2}{3^{j+1}t^2} < 1.1 \times 10^{-8}.$$
\end{proof}

Thus, assuming the validity of the Cohen-Lenstra heuristics, the probability that the conjecture for the $\mathbb{Z}_3$-extension is true is at least 99.999998\%.

\section{The $\mathbb{Z}_5$-extension}
We now study the following special case of our Main Conjecture:

\begin{conjecture}
For any positive integer $n$, the class number of $\mathbb{B}_{5, n}$ is 1.
\end{conjecture}
From the author's results on class numbers of cyclotomic fields of composite conductor \cite{MillerComp}, we can prove unconditionally that the class number of $\mathbb{B}_{5, 2}$ is 1.  Furthermore, Ichimura and Nakajima \cite{Ichimura} proved the class numbers of $\mathbb{B}_{5, n}$ are odd for all $n$, and Iwasawa \cite{Iwasawa} proved that $5$ does not divide the class number of $\mathbb{B}_{5, n}$ for all $n$.

To estimate the probability that all $\mathbb{B}_{5, n}$ have class number 1, we calculate the following upper bounds.

\begin{proposition}
There is an upper bound for the double sum
$$\sum_{\substack{q \,\mathrm{prime} \\ q \neq 2, 5}}  \sum_{j=3}^\infty \frac{\phi(5^j)}{fq^{2f}} < 0.001860$$
and an upper bound for the triple sum
$$- \sum_{\substack{q \,\mathrm{prime} \\ q \neq 2, 5}} \sum_{j=3}^\infty \sum_{k \geq 2} \frac{\phi(5^j)}{f}\log(1 - q^{-fk}) < 0.001868,$$
where $\phi$ is the Euler totient function, and $f$ is the order of $q$ modulo $5^j$.
\end{proposition}

Using this result, and assuming the validity of the Cohen-Lenstra heuristics, the probability that the conjecture for the $\mathbb{Z}_5$-extension is true is at least $$e^{-0.001868} > 99.8\%.$$

\begin{proof}
We break the double sum into two parts:
$$\sum_{\substack{q \,\mathrm{prime} \\ q \neq 2, 5}} \sum_{j=3}^\infty \frac{\phi(5^j)}{fq^{2f}} = \sum_{\substack{q \,\mathrm{prime} \\ q \neq 2, 5}} \sum_{\substack{j=3 \\ q^f \leq 10^6}}^\infty \frac{\phi(5^j)}{fq^{2f}} + \sum_{\substack{q \,\mathrm{prime} \\ q \neq 2, 5}} \sum_{\substack{j=3 \\ q^f > 10^6}}^\infty \frac{\phi(5^j)}{fq^{2f}}.$$
The first term on the right-hand side is in fact a finite sum, which we calculate explicitly
$$\sum_{\substack{q \,\mathrm{prime} \\ q \neq 2, 5}} \sum_{\substack{j=3 \\ q^f \leq 10^6}}^\infty \frac{\phi(5^j)}{fq^{2f}} < 0.001854.$$
Similarly to the argument used for the $\mathbb{Z}_2$- and $\mathbb{Z}_3$-extensions, we note that $q^f = 1 + 5^j t$ for some positive integer $t$, so the second term can be bounded by
$$ \sum_{\substack{q \,\mathrm{prime} \\ q \neq 2, 5}} \sum_{\substack{j=3 \\ q^f > 10^6}}^\infty \frac{\phi(5^j)}{fq^{2f}} < \sum_{j=3}^\infty \sum_{\substack{t = 1 \\ 1 + 5^jt > 10^6}}^\infty \frac{4(5^{j-1})}{(1+5^jt)^2} < \sum_{j=3}^\infty \sum_{\substack{t = 1 \\ 1 + 5^jt > 10^6}}^\infty \frac{4}{5^{j+1}t^2}.$$
We rearrange this to get
$$ \sum_{\substack{q \,\mathrm{prime} \\ q \neq 2, 5}} \sum_{\substack{j=3 \\ q^f > 10^6}}^\infty \frac{\phi(5^j)}{fq^{2f}} < \frac{1}{5^3} \cdot \frac{\pi^2}{6} - \sum_{j=3}^\infty \sum_{\substack{t = 1 \\ 1 + 5^jt \leq 10^6}}^\infty \frac{4}{5^{j+1}t^2} < 0.000006.$$
Thus, we have the upper bound
$$\sum_{\substack{q \,\mathrm{prime} \\ q \neq 2, 5}} \sum_{j=3}^\infty \frac{\phi(5^j)}{fq^{2f}} < 0.001854 + 0.000006 = 0.001860.$$
Using this upper bound for the double sum, we can find an upper bound for the triple sum.  Since $q^f \geq 1 + 5^3 \cdot 2 = 251$ and $q^{fk} \geq 251^2 = 63001$, we have
$$-\log(1-q^{-fk}) \leq \frac{-\log(1-1/63001)}{1/63001}q^{-fk} < 1.000008 q^{-fk}.$$
We also have that
$$\sum_{k=2}^\infty q^{-fk} = q^{-2f}(1-q^{-f})^{-1} \leq q^{-2f}(1-1/251)^{-1} \leq 1.004 q^{-2f}.$$
Thus, we have an upper bound for the triple sum
$$- \sum_{\substack{q \,\mathrm{prime} \\ q \neq 2, 5}} \sum_{j=3}^\infty \sum_{k \geq 2} \frac{\phi(5^j) \log(1 - q^{-fk})}{f} < 1.000008 \times 1.004 \times 0.001860 < 0.001868.$$
\end{proof}

\section{The $\mathbb{Z}_p$-extensions for $7 \leq p \leq 31$}
We now study the following special case of our Main Conjecture:

\begin{conjecture}
For any positive integer $n$ and any prime $p$ such that $7 \leq p \leq 31$, the class number of $\mathbb{B}_{p, n}$ is 1.
\end{conjecture}
As discussed earlier, we have proved that the class numbers of $\mathbb{B}_{p, 1}$ are 1 for $p \leq 31$ (conditional on GRH in the cases of $p=23$, 29 and 31).  Moreover, Ichimura and Nakajima \cite{Ichimura} have proven that the class numbers of $\mathbb{B}_{p, n}$  for $p \leq 500$ are odd, and  we have Iwasawa's result \cite{Iwasawa} that $p$ does not divide the class number of $\mathbb{B}_{p, n}$.

To estimate the probability that all $\mathbb{B}_{p, n}$,  $7 \leq p \leq 31$, have class number 1, we calculate the following upper bounds.

\begin{proposition}
For primes $p$ such that $7 \leq p \leq 31$, there are upper bounds for the double sums
\[\sum_{\substack{q \,\mathrm{prime} \\ q \neq 2, p}} \sum_{j=2}^\infty \frac{\phi(p^j)}{fq^{2f}} < \begin{cases}
0.001617 & \text{if } p = 7, \\
0.000896 & \text{if } p = 11, \\
0.000446 & \text{if } p = 13, \\
0.000051 & \text{if } p = 17, \\
0.000026 & \text{if } p = 19, \\
0.000017 & \text{if } p = 23, \\
0.000040 & \text{if } p = 29, \\
0.000018 & \text{if } p = 31, \\
  \end{cases}\]
and upper bounds for the triple sums
\[- \sum_{\substack{q \,\mathrm{prime} \\ q \neq 2, p}} \sum_{j=2}^\infty \sum_{k \geq 2} \frac{\phi(p^j) \log(1 - q^{-fk})}{f} < \begin{cases}
0.001626 & \text{if } p = 7, \\
0.000901 & \text{if } p = 11, \\
0.000449 & \text{if } p = 13, \\
0.000052 & \text{if } p = 17, \\
0.000027 & \text{if } p = 19, \\
0.000018 & \text{if } p = 23, \\
0.000041 & \text{if } p = 29, \\
0.000019 & \text{if } p = 31, \\
  \end{cases}.\]
where $\phi$ is the Euler totient function, and $f$ is the order of $q$ modulo $p^j$.
\end{proposition}

Using this result and assuming the validity of the Cohen-Lenstra heuristics, the conjecture that all the fields in all the $\mathbb{Z}_p$-extensons, $7 \leq p \leq 31$, have class number 1 is true with probability at least $$e^{-(0.001626+0.000901+0.000449+0.000052+0.000027+0.000018+0.000041+0.000019)} > 99.6\%.$$

\begin{proof}
We follow closely the proof used for the $\mathbb{Z}_5$-extension.  First we break the double sum into two parts,
$$\sum_{\substack{q \,\mathrm{prime} \\ q \neq 2, p}} \sum_{j=2}^\infty \frac{\phi(p^j)}{fq^{2f}} = \sum_{\substack{q \,\mathrm{prime} \\ q \neq 2, p}} \sum_{\substack{j=2 \\ q^f \leq 10^6}}^\infty \frac{\phi(p^j)}{fq^{2f}} + \sum_{\substack{q \,\mathrm{prime} \\ q \neq 2, p}} \sum_{\substack{j=2 \\ q^f > 10^6}}^\infty \frac{\phi(p^j)}{fq^{2f}}.$$
The first term on the right-hand side is a finite sum, which we  calculate explicitly.
\[\sum_{\substack{q \,\mathrm{prime} \\ q \neq 2, p}} \sum_{\substack{j=2 \\ q^f \leq 10^6}}^\infty \frac{\phi(p^j)}{fq^{2f}} < \begin{cases}
0.001611 & \text{if } p = 7, \\
0.000890 & \text{if } p = 11, \\
0.000440 & \text{if } p = 13, \\
0.000045 & \text{if } p = 17, \\
0.000020 & \text{if } p = 19, \\
0.000011 & \text{if } p = 23, \\
0.000034 & \text{if } p = 29, \\
0.000012 & \text{if } p = 31. \\
  \end{cases}\]

Since $q^f = 1 + p^j t$ for some positive integer $t$, the second term can be bounded by
$$ \sum_{\substack{q \,\mathrm{prime} \\ q \neq 2, p}} \sum_{\substack{j=2 \\ q^f > 10^6}}^\infty \frac{\phi(p^j)}{fq^{2f}} < \sum_{j=2}^\infty \sum_{\substack{t = 1 \\ 1 + p^jt > 10^6}}^\infty \frac{(p-1)(p^{j-1})}{(1+p^jt)^2} < \sum_{j=2}^\infty \sum_{\substack{t = 1 \\ 1 + p^jt > 10^6}}^\infty \frac{p-1}{p^{j+1}t^2}.$$
We rearrange this to get
$$ \sum_{\substack{q \,\mathrm{prime} \\ q \neq 2, p}} \sum_{\substack{j=2 \\ q^f > 10^6}}^\infty \frac{\phi(p^j)}{fq^{2f}} < \frac{1}{p^2} \cdot \frac{\pi^2}{6} - \sum_{j=2}^\infty \sum_{\substack{t = 1 \\ 1 + p^jt \leq 10^6}}^\infty \frac{p-1}{p^{j+1}t^2} < 0.000006,$$
which is valid for any prime $7 \leq p \leq 31$.  Thus, we have the upper bound
\[\sum_{\substack{q \,\mathrm{prime} \\ q \neq 2, p}} \sum_{j=2}^\infty \frac{\phi(p^j)}{fq^{2f}} < \begin{cases}
0.001617 & \text{if } p = 7, \\
0.000896 & \text{if } p = 11, \\
0.000446 & \text{if } p = 13, \\
0.000051 & \text{if } p = 17, \\
0.000026 & \text{if } p = 19, \\
0.000017 & \text{if } p = 23, \\
0.000040 & \text{if } p = 29, \\
0.000018 & \text{if } p = 31. \\
  \end{cases}\]

Using these upper bounds for the double sums, we can find upper bounds for the triple sums.  The smallest prime power $q^f$ that is congruent to 1 modulo $p^j$ for $7 \leq p \leq 31$ and $j \geq 2$ is $1+7^2 \cdot 4 = 197$.  Since $q^{fk} \geq 197^2 = 38809$, we have
$$-\log(1-q^{-fk}) \leq \frac{-\log(1-1/38809)}{1/38809}q^{-fk} < 1.000013 q^{-fk}.$$
We also have that
$$\sum_{k=2}^\infty q^{-fk} = q^{-2f}(1-q^{-f})^{-1} \leq q^{-2f}(1-1/197)^{-1} < 1.005103 q^{-2f}.$$
Thus, we have upper bounds for the triple sums
\[- \sum_{\substack{q \,\mathrm{prime} \\ q \neq 2, p}} \sum_{j=2}^\infty \sum_{k \geq 2} \frac{\phi(p^j) \log(1 - q^{-fk})}{f} < 1.005117 \sum_{\substack{q \,\mathrm{prime} \\ q \neq 2, p}} \sum_{\substack{j=2 \\ q^f \leq 10^6}}^\infty \frac{\phi(p^j)}{fq^{2f}}
< \begin{cases}
0.001626 & \text{if } p = 7, \\
0.000901 & \text{if } p = 11, \\
0.000449 & \text{if } p = 13, \\
0.000052 & \text{if } p = 17, \\
0.000027 & \text{if } p = 19, \\
0.000018 & \text{if } p = 23, \\
0.000041 & \text{if } p = 29, \\
0.000019 & \text{if } p = 31. \\
  \end{cases}.\]
\end{proof}

\section{First layers of the $\mathbb{Z}_p$-extensions}
The first layers $\mathbb{B}_{p,1}$ of the cyclotomic $\mathbb{Z}_p$-extensions of the rationals are of prime degree $p$.  We consider another special case of our Main Conjecture:

\begin{conjecture}
For any prime $p$, the class number of $\mathbb{B}_{p, 1}$ is 1.
\end{conjecture}
From our earlier calculations, we know that this is true for primes $p \leq 19$ and also for $p=23$, 29 and 31 under the assumption of GRH.  Assuming the Cohen-Lenstra heuristics, the probability $P_T$ that all first layers have class number 1 can be estimated by
$$P_T = \prod_{\substack{p \text{ prime} \\ p > 31}} \prod_{\substack{q \text{ prime} \\ q \neq p}} \prod_{k \geq 2} (1 - q^{-fk})^{\phi(p)/f},$$
where $f$ is the order of $q$ modulo $p$.  Taking logarithms, we get the triple sum
$$-\log{P_T} = - \sum_{p > 31} \sum_{q \neq p} \sum_{k \geq 2} \frac{p-1}{f} \log(1 - q^{-fk}).$$
Since $q^f \equiv 1 (\bmod\, p)$, we have $q^f \geq 83$ and $q^{-fk} \leq 1/83^2.$  It follows that
$$- \log(1 - q^{-fk}) \leq \frac{-\log(1-1/83^2)}{1/83^2}q^{-fk} < 1.0001 q^{-fk}.$$
This gives us an upper bound
$$E_1  < 1.0001 \sum_{p > 31} \sum_{q \neq p} \sum_{k \geq 2} \frac{p-1}{fq^{fk}} = 1.0001 \sum_{p > 31} \sum_{q \neq p} \frac{p-1}{fq^{f}(q^f - 1)} < 1.02 \sum_{p > 31} \sum_{q \neq p} \frac{p-1}{fq^{2f}}.$$
In the above double sum, the largest contribution is given by terms with $f=1$:
$$\sum_{p > 31} \sum_{\substack{q \text{ prime} \\ q \equiv 1 (\bmod\, p)}} \frac{p-1}{q^{2}} = \sum_{a \geq 1} \sum_{\substack{p > 31 \text{ prime} \\ ap+1 \text{ prime}}} \frac{p-1}{(ap+1)^{2}}.$$
The sum $\sum_{q \equiv 1(p)} (p-1)/q^2$ is $O(1/p)$, which is not good enough for convergence of the sum over $p$.  However, if we average over the various moduli $p$, we can hope for a $(\log p)^2$ savings, enough to show convergence.  To achieve such savings, we will apply an explicit sieve theory result to sums over primes of the form $ap+1$.

\section{Sums over Sophie Germain primes}
Motivated by the discussion in the previous section, we first consider sums over the \emph{Sophie Germain primes}, i.e. primes $p$ such that $2p+1$ is also prime, and more generally sums over primes $p$ such that $ap+1$ is prime, with $a$ a given fixed even integer.  Let $\pi_{ap+1}(x)$ denote the density of such primes,
$$\pi_{ap+1}(x) = \#\{p < x :  p, ap+1 \text{ both prime}\}.$$
The methods of Hardy and Littlewood give the explicit conjectural asymptotics
$$\pi_{ap+1}(x) \sim 2 C_2 \frac{x}{\log^2{x}}\prod_{p|a,\; p \neq 2} \frac{p-1}{p-2},$$
where $C_2$, the \emph{twin prime constant}, is
$$C_2 = \prod_{p > 2} \frac{p(p-2)}{(p-1)^2} \approx 0.6601618158.$$

However, to get explicit upper bounds for sums over such primes, we need explicit density estimates.  Fortunately, Siebert has provided us this key result \cite{Siebert}, which is greater than the conjectural asymptotic by a factor of 8.  Furthermore, for values of $x$ greater than $1.63 \times 10^{10}$, we can strengthen Siebert's bound by incorporating results of Riesel and Vaughan \cite[proof of Lemma 5]{RieselVaughan}.

\begin{theorem}[Siebert \cite{Siebert}, Riesel and Vaughan \cite{RieselVaughan}]
Let $a$ be an even integer.  Then there is the explicit density estimate
$$\pi_{ap+1}(x) \leq 16 C_2 \min \left( \frac{x}{(\log x)^2}, \frac{x}{(\log x)^2 + F(x)} + 2\sqrt{x} \right) \prod_{p|a,\; p \neq 2} \frac{p-1}{p-2},$$
where
$$F(x) = 8.463433 \log x - 9.260623 - 9310.077 x^{-1/6} - 29.50889 x^{-1/2}.$$
\end{theorem}

To estimate upper bounds for sums over these ``$ap+1$" primes, we can calculate a certain number of terms explicitly, and then use the density estimate to get an upper bound for the remaining terms in the ``tail" of the sum.

We start with the example of an upper bound for the sum of reciprocals of the Sophie Germain primes, a result which seems to have not previously appeared in the literature.

\begin{proposition}
The sum of the reciprocals of the Sophie Germain primes has an upper bound
$$\sum_{\substack{p \,\mathrm{ prime} \\ 2p+1 \,\mathrm{ prime}}} \frac{1}{p} < 1.88584.$$
\end{proposition}

\begin{proof}
We can calculate initial terms
$$\sum_{\substack{p < 10^{10} \text{ prime} \\ 2p+1 \text{ prime}}} \frac{1}{p} < 1.476947.$$
We use the density estimate to calculate the tail of the sum
\begin{align*}
\sum_{\substack{p > 10^{10} \text{ prime} \\ 2p+1 \text{ prime}}} \frac{1}{p} &= \int_{10^{10}}^\infty \frac{1}{x} d\pi_{2p+1}(x) < \int_{10^{10}}^\infty \frac{\pi_{2p+1}(x)}{x^2} dx \\ &\leq 16C_2 \int_{10^{10}}^\infty \min\left( \frac{1}{(\log x)^2},   \frac{1}{(\log x)^2 + F(x)} + 2x^{-1/2} \right) \frac{dx}{x},
\end{align*}
where $F(x) = 8.463433 \log x - 9.260623 - 9310.077 x^{-1/6} - 29.50889 x^{-1/2}.$
It is much easier to evaluate this integral numerically if we first make the change of variable $y = \log{x}$.
$$\sum_{\substack{p > 10^{10} \text{ prime} \\ 2p+1 \text{ prime}}} \frac{1}{p} < 16C_2 \int_{\log{10^{10}}}^\infty \min\left( \frac{1}{y^2},   \frac{1}{y^2 + F(e^y)} + 2e^{-y/2} \right) dy < 0.408894.$$
Thus, we have an upper bound
$$\sum_{\substack{p \text{ prime} \\ 2p+1 \text{ prime}}} \frac{1}{p} < 1.476947 + 0.408894 = 1.885841.$$
\end{proof}

Next is a similar result which will be used in our application to the Cohen-Lenstra heuristics.  It turns out that this sum will be the largest contributor to our estimate of the quadruple sum associated to our Main Conjecture.

\begin{proposition}
The following sum over Sophie Germain primes has an upper bound
$$ \sum_{\substack{p > 31 \,\mathrm{ prime} \\ 2p+1 \,\mathrm{ prime}}} \frac{p-1}{(2p+1)^{2}} < 0.170121.$$
\end{proposition}

\begin{proof}
We break the sum into two parts and evaluate the integral numerically to get
$$ \sum_{\substack{31 < p < 10^{10}\\ 2p+1 \text{ prime}}} \frac{p-1}{(2p+1)^{2}} + \sum_{\substack{p > 10^{10} \\ 2p+1 \text{ prime}}} \frac{p-1}{(2p+1)^{2}} < 0.067897 + \frac{1}{4} \sum_{\substack{p > 10^{10} \text{ prime} \\ 2p+1 \text{ prime}}} \frac{1}{p} < 0.170121.$$
\end{proof}

Now we calculate an upper bound for a double sum over primes $p > 31$ and primes $q$ congruent to 1 modulo $p$.

\begin{proposition}\label{DoubleSumSG}
The following double sum has an upper bound
$$ \sum_{\substack{p \,\mathrm{prime} \\ p > 31}}  \sum_{\substack{q \,\mathrm{prime} \\ q \equiv 1(\bmod p)}} \frac{p-1}{q^2} = \sum_{a \geq 2} \sum_{\substack{p > 31 \,\mathrm{ prime} \\ ap+1 \,\mathrm{ prime}}} \frac{p-1}{(ap+1)^{2}} < 0.319878.$$
\end{proposition}

\begin{proof}
We break the sum into several parts.  In the previous proposition, we calculated an estimate for the $a=2$ terms.  Next we calculate the finite sum
$$ \sum_{a = 4}^{100} \sum_{\substack{31 < p < 10^9 \\ ap+1 \text{ prime}}} \frac{p-1}{(ap+1)^{2}} < 0.047986.$$
Now let
$$G(y) = \min\left( \frac{1}{y^2},   \frac{1}{y^2 + F(e^y)} + 2e^{-y/2} \right).$$
Using the density bound for $\pi_{ap+1}$, we get an upper bound for the sum
$$ \sum_{a = 4}^{100} \sum_{\substack{p > 10^9 \\ ap+1 \text{ prime}}} \frac{p-1}{(ap+1)^{2}} < \sum_{\substack{a = 4 \\ a \text{ even}}}^{100} \left( \frac{16C_2}{a^2}\int_{\log 10^9}^\infty G(y)\,dy \prod_{p|a,\; p \neq 2} \frac{p-1}{p-2} \right) < 0.095835.$$
Finally, we again use the density bound to get an upper bound for the sum
$$ \sum_{a > 100} \sum_{\substack{p > 31 \\ ap+1 \text{ prime}}} \frac{p-1}{(ap+1)^{2}} < \sum_{\substack{a > 100 \\ a \text{ even}}} \left( \frac{16C_2}{a^2}\int_{\log 37}^\infty G(y)\,dy \prod_{p|a,\; p \neq 2} \frac{p-1}{p-2} \right).$$
To get an upper bound for the product $\prod (p-1)/(p-2)$, we observe that if an even integer $a$ has $N$ distinct odd prime factors, we have
$$\prod_{p|a,\; p \neq 2} \frac{p-1}{p-2} = \frac{p_1-1}{p_1-2}\cdot\frac{p_2-1}{p_2-2}\cdots\frac{p_N-1}{p_N-2} \leq \frac{2}{1}\cdot\frac{3}{2}\cdots\frac{N+1}{N} = N+1.$$
Since the number $N$ of odd prime divisors of $a$ is less than $\log a$, we have
$$ \sum_{a > 100} \sum_{\substack{p > 31 \\ ap+1 \text{ prime}}} \frac{p-1}{(ap+1)^{2}} < 16C_2\int_{\log 37}^\infty G(y)\,dy \sum_{\substack{a > 100 \\ a \text{ even}}}  \frac{1+\log a}{a^2} < 0.005936.$$
This gives us an explicit upper bound for the double sum
$$ \sum_{a \geq 2} \sum_{\substack{p > 31 \text{ prime} \\ ap+1 \text{ prime}}} \frac{p-1}{(ap+1)^{2}} < 0.170121 + 0.047986 + 0.095835 + 0.005936 = 0.319878.$$
\end{proof}

\section{Revisiting the first layers $\mathbb{B}_{p,1}$}
We return to the conjecture that all first layers $\mathbb{B}_{p,1}$ have class number 1.  To estimate the probability that this conjecture is true, we calculate the following upper bound.

\begin{proposition}
There is an upper bound for the double sum
$$\sum_{p > 31} \sum_{\substack{q \,\mathrm{prime} \\ q \neq p}} \frac{p-1}{fq^{2f}} < 0.321365$$
and an upper bound for the triple sum,
$$- \sum_{p > 31} \sum_{\substack{q \,\mathrm{prime} \\ q \neq p}} \sum_{k \geq 2} \frac{p-1}{f} \log(1 - q^{-fk}) < 0.327793,$$
where $f$ is the order of $q$ modulo $p$.
\end{proposition}

\begin{proof}
We first break the double sum into two pieces:
\begin{align*}
\sum_{p > 31} \sum_{\substack{q \,\mathrm{prime} \\ q \neq p}} \frac{p-1}{fq^{2f}} &= \sum_{p > 31} \sum_{\substack{q \,\mathrm{prime} \\ q \neq p, f = 1}} \frac{p-1}{fq^{2f}} + \sum_{p > 31} \sum_{\substack{q \,\mathrm{prime} \\ q \neq p, f > 1}} \frac{p-1}{fq^{2f}} \\ & \leq  \sum_{\substack{p \,\mathrm{prime} \\ p > 31}}  \sum_{\substack{q \,\mathrm{prime} \\ q \equiv 1(\bmod p)}} \frac{p-1}{q^2} + \sum_{f \geq 2} \sum_{q \text{ prime}} \frac{1}{fq^{2f}}\sum_{\substack{p > 31 \text{ prime} \\ p | q^f-1}} (p-1).
\end{align*}
In Proposition \ref{DoubleSumSG}, we proved that the first term on the right has an upper bound
$$ \sum_{\substack{p \,\mathrm{prime} \\ p > 31}}  \sum_{\substack{q \,\mathrm{prime} \\ q \equiv 1(\bmod p)}} \frac{p-1}{q^2} < 0.319878.$$
To estimate the terms with $f \geq 2$, we observe that the number of prime divisors $p > 31$ of an integer $m$ is less than $\log_{37}m < 0.28 \log m$.  Thus we have
$$\sum_{\substack{p>31 \text{ prime} \\ p | q^f-1}} (p-1) < 0.28(q^f-1)\log(q^f-1)< 0.28fq^f\log{q}.$$
Using this estimate to bound the terms with $q^f > 10^6$,
$$\sum_{f \geq 2} \sum_{q \text{ prime}} \frac{1}{fq^{2f}}\sum_{\substack{p>31 \text{ prime} \\ p | q^f-1}} (p-1) < \sum_{f \geq 2} \sum_{\substack{q \text{ prime} \\ q^f < 10^6}} \frac{1}{fq^{2f}} \sum_{\substack{p | q^f-1 \\ p > 31}}(p-1) + \sum_{f \geq 2} \sum_{\substack{q \text{ prime} \\ q^f > 10^6}} \frac{0.28\log q}{q^f},$$
we can rearrange to get
$$\sum_{f \geq 2} \sum_{q \text{ prime}} \frac{1}{fq^{2f}}\sum_{\substack{p>31 \text{ prime} \\ p | q^f-1}} (p-1) < \sum_{f \geq 2} \sum_{q \text{ prime}} \frac{0.28\log q}{q^f} \; -\sum_{\substack{f \geq 2 \\ q^f < 10^6}}\frac{1}{fq^{2f}}\left(0.28fq^f\log q - \sum_{\mathclap{\substack{p | q^f-1 \\ p > 31}}}(p-1) \right).$$
We then explicitly calculate the terms with $q^f < 10^6$
$$\sum_{\substack{f \geq 2 \\ q^f < 10^6}}\frac{1}{fq^{2f}}\left(0.28fq^f\log q - \sum_{\substack{p | q^f-1 \\ p > 31}}(p-1) \right) > 0.210016.$$
Thus, we have
$$\sum_{f \geq 2} \sum_{q \text{ prime}} \frac{1}{fq^{2f}}\sum_{\substack{p>23 \text{ prime} \\ p | q^f-1}} (p-1) < -0.210016 + 0.28 \sum_{q \text{ prime}} \frac{\log q}{q(q-1)}.$$
Since $\sum_{q \text{ prime}} (\log q) / q^2 = 0.4930911093\dots$ (see \cite{Cohen}), it is a straightforward calculation to show that
$$\sum_{q \text{ prime}} \frac{\log q}{q(q-1)} < \sum_{q < 10^6} \frac{\log q}{q(q-1)} + \frac{10^6}{10^6-1}\sum_{q > 10^6} \frac{\log q}{q^2} < 0.755367.$$
Therefore, as expected, we find that the contribution from the $f \geq 2$ terms is relatively small
$$\sum_{f \geq 2} \sum_{q \text{ prime}} \frac{1}{fq^{2f}}\sum_{\substack{p>31 \text{ prime} \\ p | q^f-1}} (p-1) < -0.210016 + 0.28 \times 0.755367 < 0.001487.$$
We conclude that our double sum has an upper bound
$$\sum_{p > 31} \sum_{\substack{q \,\mathrm{prime} \\ q \neq p}} \frac{p-1}{fq^{2f}} < 0.319878 + 0.001487 = 0.321365,$$
and that our triple sum has an upper bound
$$- \sum_{p > 31} \sum_{\substack{q \,\mathrm{prime} \\ q \neq p}} \sum_{k \geq 2} \frac{p-1}{f} \log(1 - q^{-fk}) < 1.02 \sum_{p > 31} \sum_{\substack{q \,\mathrm{prime} \\ q \neq p}} \frac{p-1}{fq^{2f}} < 0.327793.$$
\end{proof}

Let $P_T$ denote the probability (based on the Cohen-Lenstra heuristics) that all the first level fields $\mathbb{B}_{p,1}$ have class number 1.  Then, by the above proposition, we estimate that $-\log{P_T} < 0.327793$.  We can further refine our estimate by considering the results of Hakkarainen \cite{Hakkarainen}, who calculated certain divisors of the class number for abelian number fields of conductors less than 2000.  Since the fields $\mathbb{B}_{p,1}$ have conductor $p^2$, we can apply Hakkarainen's results to the fields $\mathbb{B}_{37,1}$, $\mathbb{B}_{41,1}$ and $\mathbb{B}_{43,1}$.  He showed that the class numbers of those  fields do not have any odd prime divisors less than 10000.  By excluding these terms from our sum, we can reduce our estimate for $-\log{P_T}$ by 0.010431.  Furthermore, by Theorem \ref{ParityClassNumber}, we know the parity of the class number of $\mathbb{B}_{p,1}$ is odd for $p < 500$, so we can further reduce our estimate for $-\log{P_T}$ by 0.001131.  Finally, we can incorporate unpublished results by T. Fukuda that $2p+1$ does not divide the class number of $\mathbb{B}_{p,1}$ for Sophie Germain primes less than 1000.  Fukuda's results reduce our estimate for $-\log{P_T}$ by 0.030415.  This give us an adjusted upper bound for $-\log{P_T}$ of
$$-\log{P_T} < 0.327793 - 0.010431 - 0.001131 - 0.030415 = 0.285816.$$
This gives a lower bound for $P_T$ of
$$P_T  = \exp(-0.285816) > 0.75.$$
Thus, assuming the validity of the Cohen-Lenstra heuristics, the probability that the conjecture that all the first layers $\mathbb{B}_{p,1}$ have class number 1 is at least 75\%.

\section{Higher layers of the $\mathbb{Z}_p$-extensions}
The higher layers of the $\mathbb{Z}_p$-extensions have a relatively small contribution to our estimates, and an upper bound is much easier to establish.  Let $P_T$ be the conditional probability that all the $\mathbb{B}_{p,n}$, for primes $p>31$ and layers $n \geq 2$, have class number 1, conditional upon all of the first layers $\mathbb{B}_{p,1}$ having class number 1.  We estimate $P_T$ by applying the Cohen-Lenstra heuristics
$$P_T = \prod_{\substack{p \text{ prime} \\ p > 31}} \prod_{\substack{q \text{ prime} \\ q \neq p}} \prod_{j \geq 2} \prod_{k \geq 2} (1 - q^{-fk})^{\phi(p^j)/f}.$$
We use the following proposition to estimate $-\log{P_T}$.

\begin{proposition}
The following quadruple sum has an upper bound
$$-\sum_{\substack{p \,\mathrm{ prime} \\ p > 31}} \sum_{\substack{q \,\mathrm{ prime} \\ q \neq p}} \sum_{j \geq 2} \sum_{k \geq 2} \frac{\phi(p^j)}{f} \log(1 - q^{-fk}) < 0.010398,$$
where $\phi$ is the Euler totient function and $f$ is the order of $q$ modulo $p^j$.
\end{proposition}

\begin{proof}
Let $S$ denote the quadruple sum.  Similarly to our earlier arguments, the sum $S$ has an upper bound
$$S < 1.02 \sum_{p > 31} \sum_{q \neq p} \sum_{j \geq 2} \frac{\phi(p^j)}{fq^{2f}}.$$
Since $q^f = 1 + p^j t$ for some integer $t$, we have
$$S < 1.02  \sum_{p > 31} \sum_{j \geq 2} \sum_{t \geq 1} \frac{(p-1)p^{j-1}}{(1+p^jt)^2} < 1.02  \sum_{p > 31} \sum_{j \geq 2} \sum_{t \geq 1} \frac{p-1}{p^{j+1}t^2} = 1.02\frac{\pi^2}{6} \sum_{p > 31} \frac{1}{p^2}.$$
Since $\sum_p 1/p^2 = 0.452247420041\dots$, it is straightforward to calculate that
$$S < 1.02 \frac{\pi^2}{6} \times \sum_{p>31} 1/p^2 <1.02 \frac{\pi^2}{6} \times 0.006197 < 0.010398.$$
\end{proof}

\section{Summary of results and concluding remarks}
We summarize all the preceding results in Table \ref{TableUpperBounds}, which gives upper bounds for $-\log{P_T}$, where $P_T$ is the expected probability that the associated conjecture is true, based on the Cohen-Lenstra heuristics.  In particular, for our Main Conjecture, we have that
$$-\log{P_T} = - \sum_{p \,\mathrm{prime}} \sum_{q \,\mathrm{prime}} \sum_{\mathclap{\substack{j=1 \\ \text{excluding terms where the} \\ q\text{-part of the class group} \\ \text{is known to be trivial}}}}^\infty \sum_{k \geq 2} \frac{\phi(p^j)}{f}\log(1 - q^{-fk}) < 0.301215.$$
Thus, assuming the validity of the Cohen-Lenstra heuristics, the probability that the Main Conjecture is true is at least 74\%.
$$P_T > e^{-0.301215} \approx 74\%.$$

\begin{center}
\begin{table}
\label{TableUpperBounds}
\caption{Upper bounds for $-\log{P_T}$}
\begin{tabular}{@{} l  l  l  l @{} }
\hline
  \bf{Prime} $p$ & \bf{1st level} &  \bf{Level} $> 1$  & \bf{Total}\\
\hline
   $2$     	& 0      			& $1.3 \times 10^{-8}$	& $1.3 \times 10^{-8}$ \\
   $3$     	& 0				& $1.1 \times 10^{-8}$	& $1.1 \times 10^{-8}$ \\
   $5$     	& 0             		& 0.001868			& 0.001868\\
   $7$     	& 0             		& 0.001626 			& 0.001626\\
  $11$    	& 0             		& 0.000901 			& 0.000901\\
  $13$    	& 0             		& 0.000449 			& 0.000449\\
  $17$    	& 0            			& 0.000052 			& 0.000052\\
  $19$    	& 0             		& 0.000027 			& 0.000027\\
  $23$    	& 0             		& 0.000018 			& 0.000018\\
  $29$    	& 0             		& 0.000041 			& 0.000041\\
  $31$    	& 0             		& 0.000019 			& 0.000019\\
  $p > 31$	& 0.285816		& 0.010398 			& 0.296214\\
\hline
  \bf{Total}        & 0.285816		& 0.015399			& \textbf{0.301215}\\
\hline
\end{tabular}
\end{table}
\end{center}

This is perhaps unduly pessimistic.  If we were to replace the rigorous density bounds of Siebert with the conjectural asymptotics of Hardy and Littlewood, then our estimate for $-\log{P_T}$ would be approximately 0.126, giving a probability estimate of 88\%.  This gives strong heuristic support for our conjecture that all fields in the cyclotomic $\mathbb{Z}_p$-extensions over $\mathbb{Q}$ have trivial class group.

For the more skeptically-minded, our analysis indicates the most likely places to look for possible counterexamples.  The largest contributions to our estimates arise from the 1st level $\mathbb{B}_{p,1}$ of the $\mathbb{Z}_p$-extension at the smaller Sophie Germain primes $p$; we then consider the $(2p+1)$-part of their class number.  The smallest such prime $p$ for which we do not know the $(2p+1)$-part of the class number is $p = 1013$.  Nevertheless, such primes between 1000 and 10,000 still only contribute about 0.01 to our sum, so we would expect to only have about a 1\% chance of producing a counterexample among those primes.  Similarly, for primes $p$ such that $53 \leq p \leq 10000$ and $4p+1$ is prime, we expect to have only an approximately 1\% chance of producing a counterexample.

\section{Acknowledgments}
I would like to thank my advisor, Henryk Iwaniec, for introducing me to Weber's class number problem, and his steadfast encouragement in these further explorations of the class numbers of cyclotomic fields and $\mathbb{Z}_p$-extensions.  I thank Keith Conrad, Matthew de Courcy-Ireland, Takashi Fukuda, Carl Pomerance, Jerrold Tunnell, Lawrence Washington and the anonymous referee for their careful reading and suggestions.  I would also like to gratefully acknowledge the use of SageMathCloud\textsuperscript{TM}, which is supported by the National Science Foundation under Grant No. DUE-1022574 and by Google via its Google Research Awards program.

\medskip

\end{document}